\documentclass[12pt,twoside,reqno]{amsart}
\usepackage[colorlinks=true,citecolor=blue]{hyperref}
\usepackage{mathptmx, amsmath, amssymb, amsfonts, amsthm, mathptmx, enumerate, color,mathrsfs}
\usepackage{graphicx}

\usepackage{multirow}
\usepackage{epstopdf}
\usepackage{multicol}
\usepackage{algorithm}
\usepackage{algorithmic}
\usepackage{epstopdf}

\newtheorem{theorem}{Theorem}[section]
\newtheorem{lemma}[theorem]{Lemma}

\newtheorem{corollary}[theorem]{Corollary}

\theoremstyle{definition}
\newtheorem{definition}[theorem]{Definition}

\newtheorem{remark}[theorem]{Remark}
\numberwithin{equation}{section}
\newcommand{\R}{\mathbb{R}}%
\newcommand{\N}{\mathbb{N}}%
\newcommand{\e}{\varepsilon}%
\newcommand{\ol}{\overline}%

\newcommand{\n}{{\nabla}}

\newcommand{\ds}{\displaystyle}
\newcommand{\To}{\longrightarrow}
\def\a{\alpha}

\def\b{\beta}
\def\e{\epsilon}

\def\t{\theta}
\def\g{\gamma}

\def\s{\sigma}
\def\l{\lambda}
\def\<{\langle}
\def\>{\rangle}

\DeclareMathOperator*\crit{crit}
\DeclareMathOperator*\dist{dist}

\begin{document}
\setcounter{page}{1}

\vspace*{2.0cm}
\title[A second order dynamical  system with unbounded damping]
{Solving nonconvex optimization problems via a second order dynamical  system with unbounded damping}
\author[Szil\'{a}rd Csaba L\'{a}szl\'{o}]{ Szil\'{a}rd Csaba L\'{a}szl\'{o}$^{1,*}$}
\maketitle
\vspace*{-0.6cm}

\begin{center}
{
\footnotesize

$^1$Department of Mathematics, Technical University of Cluj-Napoca,  Memorandumului 28, Cluj-Napoca, Romania
}

This paper is dedicated to the  memory of Professor Hedy Attouch

\end{center}

\vskip 4mm {\footnotesize \noindent {\bf Abstract.}
In this paper we study a second order dynamical system with variable coefficients in connection  to the minimization problem of a smooth nonconvex function. The convergence of the  trajectories generated by the dynamical system to a critical point of the objective function is assured, provided a regularization of the objective function satisfies the Kurdyka-\L{}ojasiewicz property. We also provide convergence rates for the trajectories generated by the dynamical system, formulated in terms of the \L{}ojasiewicz exponent, and we show that the unbounded damping considered in our dynamical system significantly  improves the convergence rates known so far in the literature, that is, instead of linear rates we obtain superlinear rates.


  }

\renewcommand{\thefootnote}{}
\footnotetext{ $^*$Corresponding author.
\par
E-mail address: szilard.laszlo@math.utcluj.ro.
\par
Received xx, x, xxxx; Accepted xx, x, xxxx.

\rightline {\tiny   \copyright  2022 Communications in Optimization Theory}}

\section{Introduction}\label{sec-intr}

Consider the second order dynamical system
\begin{equation}\label{dysy}\tag{DS}
\left\{
\begin{array}{lll}
\ds \frac{\l}{t^2}\ddot{x}(t)+\frac{\g}{t}\dot{x}(t)+\nabla g(x(t))=0\\
\\
\ds x(t_0)=u_0,\,\dot{x}(t_0)=v_0,
\end{array}
\right.
\end{equation}
in connection to the optimization problem
\begin{equation}\label{opt-pb}\tag{P} \ \inf_{x\in\R^n}g(x), \end{equation}
where $t_0>0,$ $g:\R^n\To \R$ is a possibly nonconvex  Fr\'{e}chet differentiable function with $L_g$-Lipschitz continuous gradient,  $u_0,v_0\in \R^n$ and $\g,\l\in (0,+\infty).$
Note that the governing differential equation of \eqref{dysy}  can be written as $\ddot{x}(t)+\frac{\g}{\l} t\dot{x}(t)+\frac{t^2}{\l}\nabla g(x(t))=0$ where the damping $\frac{\g}{\l}  t$ is indeed unbounded. Of course, instead of \eqref{dysy} one can consider the dynamical system  $\l(t)\ddot{x}(t)+\g(t)\dot{x}(t)+\nabla g(x(t))=0,\,\ds x(t_0)=u_0,\,\dot{x}(t_0)=v_0$ provided $\l(t)=\mathcal{O}\left(\frac{1}{t^2}\right)$ and $\g(t)=\mathcal{O}\left(\frac{1}{t}\right)$ as $t\to+\infty$ and the analysis presented in the paper remains valid.  We emphasize that  a particular instance of the latter system can be obtained from the dynamical system studied in \cite{beg-bolt-jen}, that is,
\begin{equation}\label{frball}\tag{BBJ}
 \ddot{y}(s)+\g\dot{y}(s)+\n g(y(s))=0,\,y(s_0)=u_0,\,\dot{y}(s_0)=v_0,\,b>0,
\end{equation} by using the time rescaling $s=t^2,\,x(t)=y(s).$
Indeed, after some easy computations we obtain
$$\frac{1}{4t^2}\ddot{x}(t)+\left(\frac{\g}{2t}-\frac{1}{4t^3}\right)\dot{x}(t)+\n g(x(t))=0,$$
therefore in some sense \eqref{dysy} and \eqref{frball} seem to be equivalent. However, according to \cite{beg-bolt-jen}, (see also \cite{BCL}), the trajectories generated by \eqref{frball} converge  to a critical point of $g$, provided $g$ is of class $C^2$ and satisfies the Kurdyka-{\L}ojasiewicz property, further the desingularizing function $\varphi$ satisfies
$$\varphi(s)\ge c\sqrt{s},\,s\in(0,\eta),$$
where $c,\eta>0$.

In contrast, in this paper we obtain the convergence of a trajectory generated by \eqref{dysy} to a critical point of $g$, provided the regularization  of $g$,
$$H:\R^n\times \R^n\To\R,\,H(u,v)=g(u)+\frac{1}{2}\|u-v\|^2$$
satisfies the Kurdyka-{\L}ojasiewicz inequality. Moreover, we will obtain some convergence rates that are better than those obtained in \cite{beg-bolt-jen} or \cite{BCL}. In particular, in case of strongly convex objective functions, we will obtain superlinear rates instead of linear rates.

Another motivation for the study of the dynamical system \eqref{dysy} arises in the following setting.
Consider the perturbed dynamical system studied in \cite{BCL1}, that is,
\begin{equation}\label{dysy10}\tag{BCL}
\ddot{x}(t)+\left(\frac{\a}{t}+\g\right)\dot{x}(t)+\nabla g(x(t))=0,\,
x(t_0)=u_0,\,\dot{x}(t_0)=v_0,
\end{equation}
where $u_0,v_0\in \R^n$ and $\a,\g\in (0,+\infty).$ Indeed, \eqref{dysy10} is a perturbed version of the dynamical system considered by Su, Boyd and Cand\`es in \cite{SBC}, the latter one can be obtained by taking $\g=0$ in \eqref{dysy10}. However the objective function in \cite{SBC} is assumed to be convex, and in this context the dynamical system studied in \cite{SBC} can be seen as the continuous counterpart of the celebrated Nesterov accelerated convex gradient method \cite{Nest1}.

By using the time rescaling $t=s^2$ and by denoting $y(s)=x\left(s^2\right),$ the system \eqref{dysy10} can be rewritten as
\begin{equation*}
\frac{1}{4 s^2}\ddot{y}(s)+\left(\frac{\g}{2s}-\frac{2\a-1}{4s^3}\right)\dot{y}(s)+\nabla g(y(s))=0,\,
y(s_0)=u_0,\,\dot{y}(s_0)=2s_0 v_0,
\end{equation*}
where $s_0=\sqrt{t_0}.$ Consequently the dynamical system \eqref{dysy} is also similar to the  dynamical system \eqref{dysy10}.

According to \cite{BCL1}, the trajectories generated by \eqref{dysy10} converge  to a critical point of $g$, provided the function $g(u)+\frac{1}{2}\|u-v\|^2,\,u,v\in\R^n$ satisfies the Kurdyka-{\L}ojasiewicz inequality.

Note that in case $g$ is a strongly convex function then according to \cite{beg-bolt-jen} and \cite{BCL1} one has
$lim_{t\to+\infty}x(t)=\ol x$
where $\ol x$ is a critical point of $g.$ Further, it holds the linear rate
$\|x(t)-\ol x\|=\mathcal{O}(e^{-At})$ as $t\to+\infty,$ where $A>0.$
In contrast, for strongly convex objective functions, in this paper we will obtain the superlinear rate
$$\|x(t)-\ol x\|=\mathcal{O}(e^{-At^2})\mbox{ as }t\to+\infty.$$
The paper is organized as follows. After presenting some preliminary notion and results that we need in order to carry out our analysis, in section 3 we show the existence and uniqueness of the trajectories of the dynamical system \eqref{dysy}. We also show that the third order derivative exists almost everywhere. In section 4 we show the convergence of the trajectories generated by the dynamical system \eqref{dysy} to a critical point of our objective function $g$. We also discuss the cases when $g$ is semialgebraic or $g$ is strongly convex.  In section 5 we obtain convergence rates in terms of the {\L}ojasiewicz exponent of the objective function. In particular we show that for strongly convex objective functions the trajectories generated by the dynamical system \eqref{dysy} converge in superlinear rate to a critical point of $g.$

\section{Preliminaries}

The finite-dimensional spaces considered in the  manuscript are endowed with the Euclidean norm topology.
The Fermat rule  in this  setting reads as: if $x\in\R^n$ is a local minimizer of $g$ then $0\in\nabla g(x)$.  We denote by
$$\crit(g)=\{x\in\R^n: 0\in\n g(x)\}$$ the set of {\it critical points} of $g$.

\begin{definition}\label{abs-cont} \rm (see, for instance, \cite{att-sv2011, abbas-att-sv}) A function $x:[0,+\infty)\rightarrow \R^n$  is said to be locally absolutely continuous, if is absolutely continuous on every interval $[0,T],\,T>0$, that is, one of the
following equivalent properties holds:

(i)  there exists an integrable function $y:[0,T]\rightarrow \R^n$ such that $$x(t)=x(0)+\int_0^t y(s)ds \ \ \forall t\in[0,T];$$

(ii) $x$ is continuous and its distributional derivative is Lebesgue integrable on $[0,T]$;

(iii) for every $\varepsilon > 0$, there exists $\eta >0$ such that for any finite family of intervals $I_k=(a_k,b_k)\subseteq [0,T]$ we have the implication
$$\left(I_k\cap I_j=\emptyset \mbox{ and }\sum_k|b_k-a_k| < \eta\right)\Longrightarrow \sum_k\|x(b_k)-x(a_k)\| < \varepsilon.$$
\end{definition}

\begin{remark}\label{rem-abs-cont}\rm (a) It follows from the definition that an absolutely continuous function is differentiable almost
everywhere, its derivative coincides with its distributional derivative almost everywhere and one can recover the function from its derivative $\dot x=y$ by the integration formula (i).

(b) If $x:[0,T]\rightarrow \R^n$ (where $T>0$) is absolutely continuous and $B:\R^n\rightarrow \R^n$ is $L$-Lipschitz continuous
(where $L\geq 0$), then the function $z=B\circ x$ is absolutely continuous, too. This can be easily seen by using the characterization of absolute continuity in
Definition \ref{abs-cont}(iii). Moreover, $z$ is almost everywhere differentiable and the inequality $\|\dot z (\cdot)\|\leq L\|\dot x(\cdot)\|$ holds almost everywhere.
\end{remark}

We recall two technical lemmas which will play an essential role in our analysis, (see  \cite{abbas-att-sv}  Lemma 5.1 and Lemma 5.2, respectively).

\begin{lemma}\label{fejer-cont1} Suppose that $F:[0,+\infty)\rightarrow\R$ is locally absolutely continuous and bounded below and that
there exists $G\in L^1([0,+\infty))$ such that for almost every $t \in [0,+\infty)$ $$\frac{d}{dt}F(t)\leq G(t).$$
Then there exists $\lim_{t\rightarrow \infty} F(t)\in\R$.
\end{lemma}

\begin{lemma}\label{fejer-cont2}  If $1 \leq p < \infty$, $1 \leq r \leq \infty$, $F:[0,+\infty)\rightarrow[0,+\infty)$ is
locally absolutely continuous, $F\in L^p([0,+\infty))$, $G:[0,+\infty)\rightarrow\R$, $G\in  L^r([0,+\infty))$ and
for almost every $t \in [0,+\infty)$ $$\frac{d}{dt}F(t)\leq G(t),$$ then $\lim_{t\rightarrow +\infty} F(t)=0$.
\end{lemma}

The convergence of the trajectories generated by the dynamical system \eqref{dysy} will be shown in the framework of the functions satisfying
the {\it Kurdyka-\L{}ojasiewicz property}. For $\eta\in(0,+\infty]$, we denote by $\Theta_{\eta}$
the class of concave and continuous functions $\varphi:[0,\eta)\rightarrow [0,+\infty)$ such that $\varphi(0)=0$, $\varphi$ is
continuously differentiable on $(0,\eta)$, continuous at $0$ and $\varphi'(s)>0$ for all
$s\in(0, \eta)$.

\begin{definition}\label{KL-property} \rm({\it Kurdyka-\L{}ojasiewicz property}) Let $f:\R^n\rightarrow \R$ be a differentiable function. We say that $f$ satisfies the {\it Kurdyka-\L{}ojasiewicz (KL) property} at
$\ol x\in \R^n$
if there exist $\eta \in(0,+\infty]$, a neighborhood $U$ of $\ol x$ and a function $\varphi\in \Theta_{\eta}$ such that for all $x$ in the
intersection
$$U\cap \{x\in\R^n: f(\ol x)<f(x)<f(\ol x)+\eta\}$$ the following inequality holds
$$\varphi'(f(x)-f(\ol x))\|\n f(x))\|\geq 1.$$
If $f$ satisfies the KL property at each point in $\R^n$, then $f$ is called a {\it KL function}.
\end{definition}

The function in the above definition is called a desingularizing function \cite{beg-bolt-jen}. The origins of this notion go back to the pioneering work of \L{}ojasiewicz \cite{lojasiewicz1963}, where it is proved that for a
real-analytic function $f:\R^n\rightarrow\R$ and a critical point $\ol x\in\R^n$ (that is $\nabla f(\ol x)=0$), there exists $\theta\in[1/2,1)$ such that the function
$|f-f(\ol x)|^{\theta}\|\nabla f\|^{-1}$ is bounded around $\ol x$. This corresponds to the situation when
$\varphi(s)=C(1-\theta)^{-1}s^{1-\theta}$. The result of \L{}ojasiewicz allows the interpretation of the KL property as a re-parametrization of the function values in order to avoid flatness around the
critical points. Kurdyka \cite{kurdyka1998} extended this property to differentiable functions definable in an o-minimal structure.
Further extensions to the nonsmooth setting can be found in \cite{b-d-l2006, att-b-red-soub2010, b-d-l-s2007, b-d-l-m2010}.

One of the remarkable properties of the KL functions is their ubiquity in applications, see \cite{b-sab-teb, BCL1,ALV-numa,L-mapr,L-jota}. To the class of KL functions belong semi-algebraic, real sub-analytic, semiconvex, uniformly convex and
convex functions satisfying a growth condition. We refer the reader to
\cite{b-d-l2006, att-b-red-soub2010, b-d-l-m2010, b-sab-teb, b-d-l-s2007, att-b-sv2013, attouch-bolte2009} and the references therein  for more details regarding all the classes mentioned above and illustrating examples.

An important role in our convergence analysis will be played by the following uniformized KL property given in \cite[Lemma 6]{b-sab-teb}.
We recall that the {\it distance function}
to a set is defined for $A\subseteq\R^n$ as $\dist(x,A)=\inf_{y\in A}\|x-y\|$ for all $x\in\R^n$.
\begin{lemma}\label{unif-KL-property} Let $\Omega\subseteq \R^n$ be a compact set and let $f:\R^n\rightarrow \R$ be a differentiable function. Assume that $f$ is constant on $\Omega$ and $f$ satisfies the KL property at each point of $\Omega$.
Then there exist $\varepsilon,\eta >0$ and $\varphi\in \Theta_{\eta}$ such that for all $\ol x\in\Omega$ and for all $x$ in the intersection
\begin{equation*} \{x\in\R^n: \dist(x,\Omega)<\varepsilon\}\cap \{x\in\R^n: f(\ol x)<f(x)<f(\ol x)+\eta\}\end{equation*}
the following inequality holds
\begin{equation*}\varphi'(f(x)-f(\ol x))\|\n f(x)\|\geq 1.\end{equation*}
\end{lemma}

A related notion that we need is the notion of a {\L}ojasiewicz exponent, which is defined   as follows, see \cite{att-b-red-soub2010,att-b-sv2013,LiP}.

\begin{definition}\label{Lexpo}  Let $f:\R^n\To \R $ be a differentiable function. The function $f$ is said to fulfill the {\L}ojasiewicz property, if for every $\ol x\in\crit{f}$ there exist $K,\e>0$ and $\t\in(0,1)$ such that
$$|f(x)-f(\ol x)|^\t\le K\|\n f(x)\|\mbox{ for every }x\mbox{ fulfilling }\|x-\ol x\|<\e.$$

The number $\t$  is called the {\L}ojasiewicz exponent of $f$ at the critical  point $\ol x.$
\end{definition}

\section{Existence and uniqueness}

Let us specifying which type of solutions are we considering in the analysis of the dynamical system \eqref{dysy}.

\begin{definition}\label{str-sol}\rm We say that $x:[t_0,+\infty)\rightarrow \R^n$ is a strong global solution of \eqref{dysy}, if the
following properties are satisfied:

(i) $x,\dot x,\ddot x:[t_0,+\infty)\rightarrow \R^n$ are locally absolutely continuous, in other words, absolutely continuous on each interval $[0,b]$ for $0<b<+\infty$;

(ii) $\frac{\l}{t^2}\ddot{x}(t)+\frac{\g}{t}\dot{x}(t)+\nabla g(x(t))=0$ for almost every $t\geq t_0$;

(iii) $x(t_0)=u_0$ and $\dot{x}(t_0)=v_0$.
\end{definition}

For proving existence and uniqueness of the strong global solutions of \eqref{dysy}, we use the Cauchy-Lipschitz-Picard Theorem for absolutely continues trajectories (see for example
\cite[Proposition 6.2.1]{haraux}, \cite[Theorem 54]{sontag}). The key argument is that one can rewrite \eqref{dysy} as a particular first order dynamical system in a suitably chosen product space (see also \cite{alv-att-bolte-red}).

\begin{theorem}\label{uniq} Assume that $\l,\g>0$. Then, for every starting points $u_0,v_0\in \R^n$ there exists a unique strong global solution $x$ of the dynamical system \eqref{dysy}.
\end{theorem}
\begin{proof}

By making use of the notation $X(t)=(x(t),\dot{x}(t))$ the system \eqref{dysy} can be rewritten as a first order dynamical system.
\begin{equation}\label{dysy4}
\left\{
\begin{array}{ll}
\dot{X}(t)=F(t,X(t))\\
X(s_0)=(u_0,v_0),
\end{array}
\right.
\end{equation}
where $F:[t_0,+\infty)\times \R^n\times\R^n\To \R^n\times\R^n,\, F(t,u,v)=\left(v, -\frac{\g t}{\l}v-\frac{t^2}{\l}\nabla g(u)\big)\right).$
First we show that $F(t,\cdot,\cdot)$ is Lipschitz continuous with a Lipschitz constant $L(t)\in L^1_{loc}([t_0,+\infty))$, for every $t\geq t_0$.

Indeed,
$$\|F(t,u,v)-F(t,\ol u,\ol v)\|=\sqrt{\|v-\ol v\|^2+\left\|\frac{\g t}{\l}(\ol v-v)+\frac{t^2}{\l}(\nabla g(\ol u)-\nabla g(u))\right\|^2}\le$$
$$\sqrt{\left(1+2\frac{\g^2 t^2}{\l^2}\right)\|v-\ol v\|^2+\frac{2L_g^2 t^4}{\l^2}\|u-\ol u\|^2}\le$$
$$\sqrt{1+2\frac{\g^2 t^2+L_g^2 t^4}{\l^2}}\sqrt{\|v-\ol v\|^2+\|u-\ol u\|^2}=$$
$$\sqrt{1+2\frac{\g^2 t^2+L_g^2 t^4}{\l^2}}\|(u,v)-(\ol u,\ol v)\|.$$

Obviously, the Lipschitz constant $L(t)=\sqrt{1+2\frac{\g^2 t^2+L_g^2 t^4}{\l^2}}$ is continuous, hence integrable on $[t_0,b]$ for all $0<b<+\infty$, consequently $L(t)\in L^1_{loc}([t_0,+\infty))$.

Next we show that $F(\cdot,u,v)\in L^1_{loc}([t_0,+\infty),{\R^n}\times {\R^n})$ for all $u,v\in {\R^n}.$

Indeed, for $0<b<+\infty$ and $u,v\in {\R^n}$ fixed one has
$$\int_{t_0}^b\|F(t,u,v)\|ds=\int_{t_0}^b \sqrt{\|v\|^2+\left\|\frac{\g t}{\l}v-\frac{t^2}{\l}\nabla g(u)\right\|^2}dt\le$$
$$\int_{t_0}^b \sqrt{\left(1+2\frac{\g^2 t^2}{\l^2}\right)\|v\|^2+2\frac{t^4}{\l^2}\|\nabla g(u)\|^2}dt\le$$
$$\sqrt{\|v\|^2+\|\nabla g(u)\|^2}\int_{t_0}^b \sqrt{1+2\frac{\g^2 t^2+t^4}{\l^2}}dt,$$
and the conclusion follows by the continuity of $\sqrt{1+2\frac{\g^2 t^2+t^4}{\l^2}}.$

The conclusion of the theorem follows by applying the Cauchy-Lipschitz-Picard theorem to the first order dynamical system \eqref{dysy4}.
\end{proof}

Considering again the setting of the proof of Theorem \ref{uniq}, from Remark \ref{rem-abs-cont}(b)
it follows that $\ddot{X}$ exists almost everywhere on $[t_0,+\infty).$ We have the following result.

\begin{lemma}\label{thirdderiv11} Consider $u_0,v_0\in\R^n$ and let $x$ be the unique strong global solution of \eqref{dysy}. Then, for almost every $t\in[t_0,+\infty)$ one has:
\begin{align}\label{third-deriv11}
\|x^{(3)}(t)\|\le \left(3\frac{\g}{\l}+\frac{L_g t^2}{\l}\right)\|\dot{x}(t)\|+
\left(\frac{\g}{\l}t+\frac{2}{t}\right)\|\ddot{x}(t)\|.
\end{align}
\end{lemma}
\begin{proof}
By using the 1-norm of $\R^n\times \R^n$ we have,

$$\|\dot{X}(t+h)-\dot{X}(t)\|=\|F(t+h,X(t+h))-F(t,X(t))\|=$$
$$\left\|(\dot{x}(t+h)-\dot{x}(t),-\frac{\g}{\l}(t+h)\dot{x}(t+h)+\frac{\g}{\l}t\dot{x}(t)-\frac{(t+h)^2}{\l}\nabla g(x(t+h))+\frac{t^2}{\l}\nabla g(x(t)))\right\|=$$
$$\|\dot{x}(t+h)-\dot{x}(t)\|+\left\|-\frac{\g}{\l}(t+h)\dot{x}(t+h)+\frac{\g}{\l}t\dot{x}(t)-\frac{(t+h)^2}{\l}\nabla g(x(t+h))+\frac{t^2}{\l}\nabla g(x(t))\right\|\le$$
$$\|\dot{x}(t+h)-\dot{x}(t)\|+\frac{\g}{\l}\left\|(t+h)\dot{x}(t+h)-t\dot{x}(t)\right\|+\frac{1}{\l}\left\|(t+h)^2\nabla g(x(t+h))-t^2\nabla g(x(t))\right\|\le$$
$$\|\dot{x}(t+h)-\dot{x}(t)\|+\frac{\g}{\l}\left\|(t+h)\dot{x}(t+h)-t\dot{x}(t)\right\|+\frac{1}{\l}\left\|((t+h)^2-t^2)\nabla g(x(t+h))\right\|+$$
$$\frac{t^2}{\l}\left\|\nabla g(x(t+h))-\nabla g(x(t))\right\|\le$$
$$\|\dot{x}(t+h)-\dot{x}(t)\|+\frac{\g}{\l}\left\|(t+h)\dot{x}(t+h)-t\dot{x}(t)\right\|+\frac{1}{\l}\left\|((t+h)^2-t^2)\nabla g(x(t+h))\right\|+$$
$$\frac{L_g t^2}{\l}\left\|x(t+h)-x(t)\right\|.$$

Hence,
$$\left\|\frac{\dot{X}(t+h)-\dot{X}(t)}{h}\right\|\le$$
$$\left\|\frac{\dot{x}(t+h)-\dot{x}(t)}{h}\right\|+\frac{\g}{\l}\left\|\frac{(t+h)\dot{x}(t+h)-t\dot{x}(t)}{h}\right\|+\frac{1}{\l}\left\|\frac{(t+h)^2-t^2}{h}\nabla g(x(t+h))\right\|+$$
$$\frac{L_g t^2}{\l}\left\|\frac{x(t+h)-x(t)}{h}\right\|.$$

By taking the limit $h\To 0$ we obtain

$$\|\ddot{X}(t)\|\le\|\ddot{x}(t)\|+\frac{\g}{\l}\left\|\left(t\dot{x}(t)\right)'\right\|+\frac{2t}{\l}\|\nabla g(x(t))\|+\frac{L_g t^2}{\l}\|\dot{x}(t)\|.$$

Now $\|\ddot{X}(t)\|=\|x^{(3)}(t)\|+\|\ddot{x}(t)\|$ and by using \eqref{dysy} we get $\nabla g(x(t))=-\frac{\l}{t^2}\ddot{x}(t)-\frac{\g}{t}\dot{x}(t)$, consequently
$$
\|x^{(3)}(t)\|\le \frac{\g}{\l}\left\|\dot{x}(t)+t\ddot{x}(t)\right\|+\frac{2t}{\l}\left\|\frac{\l}{t^2}\ddot{x}(t)+\frac{\g}{t}\dot{x}(t)\right\|+\frac{L_g t^2}{\l}\|\dot{x}(t)\|.
$$

Hence,

$$
\|x^{(3)}(t)\|\le \left(3\frac{\g}{\l}+\frac{L_g t^2}{\l}\right)\|\dot{x}(t)\|+
\left(\frac2t+\frac{\g}{\l}t\right)\|\ddot{x}(t)\|.$$
\end{proof}

\section{Convergence of trajectories}

In this section we study thee convergence of the trajectories of the dynamical system \eqref{dysy}. 

The set of limit points of the trajectory $x,$ which we denote by  $\omega(x)$, is defined as
$$\omega(x):=\{\ol{x}\in\R^n:\exists t_k\To\infty\mbox{ such that }x(t_k)\To\ol{x},\,k\To + \infty\}.$$

\begin{theorem}\label{abstr}
 Assume that $g$ is bounded from below and for $u_0,v_0\in\R^n$ let $x$ be the unique strong global solution of  the dynamical system \eqref{dysy}.

Then

\begin{itemize}
\item[(i)] $\frac{1}{\sqrt{t^3}}\ddot{x}(t),\,\frac{1}{\sqrt{t}}\dot{x}(t)\in L^2([t_0,+\infty),\R^n)$;
\item[(ii)] there exists and is finite the limit $\lim_{t\To+\infty}g(x(t))$;
\item[(iii)] $\lim_{t\To+\infty}\frac{1}{t^2}\ddot{x}(t)=0,$ and $\lim_{t\To+\infty}\frac{1}{t}\dot{x}(t)=0$;
\item[(iv)] $\omega(x)\subseteq \crit(g).$
\end{itemize}
\end{theorem}
 \begin{proof}
Let $T>0.$ Since $x,\dot{x},\ddot{x}$ are locally absolutely continuous we have $x,\dot{x},\ddot{x}\in L^2([0,T],\R^n).$ Further, by the $L_g$-Lipschitz property of $\n g$ we have $\n (g(x))\in L^2([0,T],\R^n).$ 

Consider $\b$ such that  $0<\b<\min\left(\frac{2\g}{3L_g},\sqrt{\frac{2\l}{L_g}}\right).$
By using the $L_g$ Lipschitz continuity of $\n g$ we evaluate:
\begin{align*}
\frac{d}{dt} g\left(\frac{\b}{t}\dot{x}(t)+x(t)\right)&=
\left\<\frac{\b}{t}\ddot{x}(t)+\left(1-\frac{\b}{t^2}\right)\dot{x}(t),\n g\left(\frac{\b}{t}\dot{x}(t)+x(t)\right)-\n g(x(t))\right\>\\
&+\left\<\frac{\b}{t}\ddot{x}(t)+\left(1-\frac{\b}{t^2}\right)\dot{x}(t),-\frac{\l}{t^2}\ddot{x}(s)-\frac{\g}{t}\dot{x}(t)\right\>\\
&\le
-\frac{\b\l}{t^3}\|\ddot{x}(t)\|^2-\left(\frac{\b\g+\l}{t^2}-\frac{\b\l}{t^4}\right)\<\ddot{x}(t),\dot{x}(t)\>-\left(\frac{\g}{t}-
\frac{\g\b}{t^3}\right)\|\dot{x}(t)\|^2\\
&+L_g\left\|\frac{\b}{t}\ddot{x}(t)+\left(1-\frac{\b}{t^2}\right)\dot{x}(t)\right\|\left\|\frac{\b}{t}\dot{x}(t)\right\|.
\end{align*}

Further by using  the inequality $\|\dot{x}(t)\|\|\ddot{x}(t)\|\le\frac12\left(\|\dot{x}(t)\|^2+\|\ddot{x}(t)\|^2\right)$ we have:
\begin{align*}
\left\|\frac{\b}{t}\ddot{x}(t)+\left(1-\frac{\b}{t^2}\right)\dot{x}(t)\right\|\left\|\frac{\b}{t}\dot{x}(t)\right\|&\le
\left\|\frac{\b}{t}\ddot{x}(t)\right\|\left\|\frac{\b}{t}\dot{x}(t)\right\|+\left|1-\frac{\b}{t^2}\right|\frac{\b}{t}\|\dot{x}(t)\|^2\\
&=\left\|\sqrt{\left(\frac{\b}{t}\right)^3}\ddot{x}(t)\right\|\left\|\sqrt{\frac{\b}{t}}\dot{x}(t)\right\|+\left|1-\frac{\b}{t^2}\right|
\frac{\b}{t}\|\dot{x}(t)\|^2\\
&\le\frac12\left(\left(\frac{\b}{t}\right)^3\|\ddot{x}(t)\|^2+\frac{\b}{t}\|\dot{x}(t)\|^2\right)+\left|1-\frac{\b}{t^2}\right|
\frac{\b}{t}\|\dot{x}(t)\|^2.
\end{align*}

On the other hand
\begin{align*}-\left(\frac{\b\g+\l}{t^2}-\frac{\b\l}{t^4}\right)\<\ddot{x}(t),\dot{x}(t)\>&=
-\frac12\frac{d}{dt}\left(\left(\frac{\b\g+\l}{t^2}-\frac{\b\l}{t^4}\right)\|\dot{x}(t)\|^2\right)\\
&-\left(\frac{\l+\b\g}{t^3}-\frac{2\b\l}{t^5}\right)\|\dot{x}(t)\|^2.
\end{align*}

Consequently, one has:

\begin{align}\label{e11}
\frac{d}{dt} \left(g\left(\frac{\b}{t}\dot{x}(t)+x(t)\right)+\frac12\left(\frac{\b\g+\l}{t^2}-\frac{\b\l}{t^4}\right)\|\dot{x}(t)\|^2\right)\le
\end{align}
$$\left(-\frac{\b\l}{t^3}+\frac{L_g\b^3}{2t^3}\right)\|\ddot{x}(t)\|^2+\left(-\frac{\g}{t}+\frac{L_g\b}{2t}+L_g\left|1-\frac{\b}{t^2}\right|
\frac{\b}{t}-\frac{\l}{t^3}+\frac{2\b\l}{t^5}\right)\|\dot{x}(t)\|^2.$$
Let us denote
$$A(t)=-\frac{\b\l}{t^3}+\frac{L_g\b^3}{2t^3}$$
and
$$B(t)=-\frac{\g}{t}+\frac{L_g\b}{2t}+L_g\left|1-\frac{\b}{t^2}\right|\frac{\b}{t}-\frac{\l}{t^3}+\frac{2\b\l}{t^5}.$$

Since $0<\b<\min\left(\frac{2\g}{3L_g},\sqrt{\frac{2\l}{L_g}}\right)$ we have that there exists $t_1\in\R$ such that for all $t>t_1$ one has
$$\frac{\b\g+\l}{t^2}-\frac{\b\l}{t^4}>0,\,A(t)<0\mbox{ and }B(t)<0.$$
 Now by integrating \eqref{e11} on  $[t_1,T]$ one obtains:

\begin{align}\label{e12} g\left(\frac{\b}{T}\dot{x}(T)+x(T)\right)+\frac12\left(\frac{\b\g+\l}{T^2}-\frac{\b\l}{T^4}\right)\|\dot{x}(T)\|^2+
\int_{t_1}^T -A(t)\|\ddot{x}(t)\|^2dt+\int_{t_1}^T -B(t)\|\dot{x}(t)\|^2dt\le
\end{align}
$$ g\left(\frac{\b}{t_1}\dot{x}(t_1)+x(t_1)\right)+\frac12\left(\frac{\b\g+\l}{t_1^2}-\frac{\b\l}{t_1^4}\right)\|\dot{x}(t_1)\|^2.$$

Taking into account that $g$ is bounded from bellow, by taking the limit as $T\To+\infty$ we obtain,  that
$$\int_{t_1}^\infty -A(t)\|\ddot{x}(t)\|^2dt<+\infty$$
and
$$\int_{t_1}^\infty -B(t)\|\dot{x}(t)\|^2dt<+\infty.$$

Consequently
$-A(t)\|\ddot{x}(t)\|^2,\,-B(t)\|\dot{x}(t)\|^2\in L^1([t_0,+\infty),\R)$, that is
$$\sqrt{|A(t)|}\ddot{x}(t),\,\sqrt{|B(t)|}\dot{x}(t)\in L^2([t_0,+\infty),\R^n).$$
In other words,
$$\frac{1}{\sqrt{t^3}}\ddot{x}(t),\,\frac{1}{\sqrt{t}}\dot{x}(s)\in L^2([t_0,+\infty),\R^n).$$
Further, \eqref{e11} and Lemma \ref{fejer-cont1} assure that there exists and is finite the limit
\begin{align}\label{e13}
\lim_{t\To+\infty}\left(g\left(\frac{\b}{t}\dot{x}(t)+x(t)\right)+\frac12\left(\frac{\b\g+\l}{t^2}-\frac{\b\l}{t^4}\right)\|\dot{x}(t)\|^2\right).
\end{align}

We show next that $\lim_{t\To+\infty}\frac{1}{t}\dot{x}(t)=0.$

We have:
$$\frac{d}{dt}\left(\frac{1}{t^2}\|\dot{x}(t)\|^2\right)=-\frac{2}{t^3}\|\dot{x}(t)\|^2+\frac{2}{t^2}\<\ddot{x}(t),\dot{x}(t)\>=
-\frac{2}{t^3}\|\dot{x}(t)\|^2+2\left\<\frac{1}{\sqrt{t^3}}\ddot{x}(t),\frac{1}{\sqrt{t}}\dot{x}(t)\right\>\le$$
$$-\frac{2}{t^3}\|\dot{x}(t)\|^2+\frac{1}{t^3}\|\ddot{x}(t)\|^2+\frac{1}{t}\|\dot{x}(t)\|^2\in L^1([t_0,+\infty)).$$
The conclusion follows by Lemma \ref{fejer-cont2}.

Similarly, we show that $\lim_{t\To+\infty}\frac{1}{t^2}\ddot{x}(t)=0.$

We have:
$$\frac{d}{dt}\left(\frac{1}{t^4}\|\ddot{x}(t)\|^2\right)=-\frac{4}{t^5}\|\ddot{x}(t)\|^2+\frac{2}{t^4}\<{x}^{(3)}(t),\ddot{x}(t)\>\le -\frac{4}{t^5}\|\ddot{x}(t)\|^2+\frac{2}{t^4}\|x^{(3)}(t)\|\|\ddot{x}(t)\|.$$
Now according to Lemma \ref{thirdderiv11} we have
$$\|x^{(3)}(t)\|\le \left(3\frac{\g}{\l}+\frac{L_g t^2}{\l}\right)\|\dot{x}(t)\|+
\left(\frac{\g}{\l}t+\frac{2}{t}\right)\|\ddot{x}(t)\|,$$
hence,
$$\frac{d}{dt}\left(\frac{1}{t^4}\|\ddot{x}(t)\|^2\right)\le
-\frac{4}{t^5}\|\ddot{x}(t)\|^2+\left(6\frac{\g}{\l}\frac{1}{t^4}+\frac{2L_g }{\l}\frac{1}{t^2}\right)\|\dot{x}(t)\|\|\ddot{x}(t)\|+
\left(\frac{2\g}{\l}\frac{1}{t^3}+\frac{4}{t^5}\right)\|\ddot{x}(t)\|^2\le$$
$$\frac{2\g}{\l}\frac{1}{t^3}\|\ddot{x}(t)\|^2+3\frac{\g}{\l}\frac{1}{t^4}\|\dot{x}(t)\|^2+3\frac{\g}{\l}\frac{1}{t^4}\|\ddot{x}(t)\|^2+
\frac{2L_g }{\l}\left\|\frac{1}{\sqrt{t}}\dot{x}(t)\right\|\left\|\frac{1}{\sqrt{t^3}}\ddot{x}(t)\right\|\le$$
$$\frac{2\g}{\l}\frac{1}{t^3}\|\ddot{x}(t)\|^2+3\frac{\g}{\l}\frac{1}{t^4}\|\dot{x}(t)\|^2+3\frac{\g}{\l}\frac{1}{t^4}\|\ddot{x}(t)\|^2+
\frac{L_g }{\l}\left(\frac{1}{t}\|\dot{x}(t)\|^2+\frac{1}{t^3}\|\ddot{x}(t)\|^2\right)\in L^1([t_0,+\infty)).$$
The conclusion follows by Lemma \ref{fejer-cont2}.

Finally, by using the continuity of $g$ and the fact that $\lim_{t\To+\infty}\frac{1}{t}\dot{x}(t)=0,$ \eqref{e13} becomes:
\begin{equation}\label{e14}\exists \lim_{t\To+\infty}\left(g\left(\frac{\b}{t}\dot{x}(t)+x(t)\right)+\frac12\left(\frac{\b\g+\l}{t^2}-\frac{\b\l}{t^4}\right)\|\dot{x}(t)\|^2\right)=
\lim_{t\To+\infty}g(x(t))\in\R.
\end{equation}

Let $\ol x\in\omega(x).$ Then, there exists a sequence $t_k\To+\infty,\,k\To+\infty$ such that $x(t_k)\To\ol x,\,k\To+\infty.$
Hence, by using the continuity of $\n g$ and the dynamical system \eqref{dysy}, we have
$$0=\lim_{k\To+\infty}\left(\frac{\l}{t_k^2}\ddot{x}(t_k)+\frac{\g}{t_k}\dot{x}(t_k)+\n g(x(t_k))\right)=\n g(\ol x),$$
which shows that $\ol x\in\crit(g).$
\end{proof}

The following result is an easy consequence of Theorem \ref{abstr}.

\begin{lemma}\label{reg} Assume that $g$ is bounded from below and for $u_0,v_0\in\R^n$ let $x$ be the unique strong global solution of  the dynamical system \eqref{dysy}. Consider the function
$$H:\R^n\times\R^n\To\R,\,H(x,y)=g(x)+\frac12\|x-y\|^2.$$

Consider $0<\b<\min\left(\frac{2\g}{3L_g},\sqrt{\frac{2\l}{L_g}}\right).$ For simplicity, for $t\ge t_1$, where $t_1$ was defined in Theorem \ref{abstr}, let us denote
$$u(t)=\frac{\b}{t}\dot{x}(t)+x(t)\mbox{ and }v(t)=\left(\sqrt{\frac{\b\g+\l}{t^2}-\frac{\b\l}{t^4}}+\frac{\b}{t}\right)\dot{x}(t)+x(t).$$

Then
\begin{itemize}
\item[(i)] $\omega(u)=\omega(v)=\omega(x);$
\item[(ii)]$\frac{d}{dt}H(u(t),v(t))\le A(t)\|\ddot{x}(t)\|^2+B(t)\|\dot{x}(t)\|^2\le 0,\,\mbox{ for all }t\ge t_1;$
\item[(iii)] there exists the limit $\lim_{t\To+\infty}H(u(t),v(t))=\lim_{t\To+\infty}g(x(t))\in\R;$
\item[(iv)] $H$ is finite and constant on $\omega(u,v)=\{(\ol x,\ol x)\in\R^n\times\R^n:\ol x\in\omega(x)\};$
\item[(v)] $\|\nabla H(u(t),v(t))\|\le q(t)\|\ddot{x}(t)\|+p(t)\|\dot{x}(t)\|,\,\,\mbox{ for all }t\ge t_1;$
\item[(vi)] $\omega(u,v)\subseteq \crit(H),$
\end{itemize}
where
$$A(t)=-\frac{\b\l}{t^3}+\frac{L_g\b^3}{2t^3},\,B(t)=-\frac{\g}{t}+\frac{L_g\b}{2t}+L_g\left|1-\frac{\b}{t^2}\right|\frac{\b}{t}-\frac{\l}{t^3}+\frac{2\b\l}{t^5},$$
$$q(t)=\frac{\l}{t^2}\mbox{ and }p(t)=\frac{L_g\b}{t}+\frac{\g}{t}+2\sqrt{\frac{\b\g+\l}{t^2}-\frac{\b\l}{t^4}}.$$

Assume further that $x$ is bounded. Then
\begin{itemize}
\item[(vii)] $\omega(u,v)$ is nonempty and compact;
\item[(viii)] $\lim_{t\To+\infty}\dist((u(t),v(t)),\omega(u,v))=0.$
\end{itemize}

\end{lemma}
\begin{proof} (i) According to Theorem \ref{abstr} (i)
$$\lim_{t\To+\infty}\frac{\b}{t}\dot{x}(t)=\lim_{t\To+\infty}\left(\sqrt{\frac{\b\g+\l}{t^2}-\frac{\b\l}{t^4}}+\frac{\b}{t}\right)\dot{x}(t)=0,$$
hence $$\omega(u)=\omega(v)=\omega(x).$$

(ii) and (iii) are nothing else than \eqref{e11} and \eqref{e14}.

(iv) Consider $(\ol x,\ol y)\in\omega(u,v)$. Then, there exists a sequence $t_k\To+\infty,\,k\To+\infty$ such that $u(t_k)\To\ol x,\,k\To+\infty\mbox{ and }v(t_k)\To\ol y,\,k\To+\infty.$ But according to Theorem \ref{abstr} (i) and the form of $u$ and $v$ we get that $\ol x=\ol y$ and $\ol x\in\omega(x).$ Consequently, $\omega(u,v)=\{(\ol x,\ol x)\in\R^n\times\R^n:\ol x\in\omega(x)\}.$

Now, let $(\ol x,\ol x)\in\omega(u,v)$. Then, there exists a sequence $t_k\To+\infty,\,k\To+\infty$ such that $u(t_k)\To\ol x,\,k\To+\infty\mbox{ and }v(t_k)\To\ol x,\,k\To+\infty.$
But according to (iii) $\lim_{t\To+\infty}H(u(t),v(t))=l\in\R$, hence $H(\ol x,\ol x)=\lim_{t_k\To+\infty}H(u(t_k),v(t_k))=\lim_{t\To+\infty}H(u(t),v(t))=l.$ Consequently, $H$  is constant on $\omega(u,v).$

(v) Since $\nabla H(x,y)=(\nabla g(x)+ x-y, y-x)$ by using \eqref{dysy}, we have

$$\|\nabla H(u(t),v(t))\|\le\|\nabla g(u(t))\|+2\|u(t)-v(t)\|\le\|\nabla g(u(t))-\nabla g(x(t))\|+\|\nabla g(x(t))\|+2\|u(t)-v(t)\|\le$$
$$\frac{L_g\b}{t}\|\dot{x}(t)\|+\left\|-\frac{\l}{t^2}\ddot{x}(t)-\frac{\g}{t}\dot{x}(t)\right\|+2\|u(t)-v(t)\|\le$$
$$\frac{\l}{t^2}\|\ddot{x}(t)\|+\left(\frac{L_g\b}{t}+\frac{\g}{t}+2\sqrt{\frac{\b\g+\l}{t^2}-\frac{\b\l}{t^4}}\right)\|\dot{x}(t)=$$
$$q(t)\|\ddot{x}(t)\|+p(t)\|\dot{x}(t)\|.$$

(vi) Since $\crit{H}=\{(x,y)\in\R^n\times\R^n:(0,0)\in\n H(x,y)\}=\{(\ol x,\ol x)\in\R^n\times\R^n:\ol x\in\crit(g)\},$ and
according to (i) $\omega(u,v)=\{(\ol x,\ol x)\in\R^n\times\R^n:\ol x\in\omega(x)\}$ by Theorem \ref{abstr} (iv), one has:
$$\omega(u,v)\subseteq \crit(H).$$

(vii) Since $x$ is bounded and $\frac{1}{t}\dot{x}(t)\To 0,\,t\To+\infty$ we obtain that $u$ and $v$ are bounded, respectively. Hence from Bolzano-Weierstrass Theorem one obtains that the set of limit points $\omega(u,v)$ is nonempty. Moreover
$\omega(u,v)=\{(\ol x,\ol x)\in\R^n\times\R^n:\ol x\in\omega(x)\}$, and $\omega(x)$ is bounded, thus it is enough to show that $\omega(x)$ is closed.

Consider $(\ol x_n)_{n\ge 1}\subseteq \omega(x)$ and assume that $\lim_{n\To+\infty}\ol x_n=x^*.$ We show that $x^*\in\omega(x).$

Obviously, for every $n\ge 1$ there exits the sequence $t_k^n\To+\infty,\,k\To+\infty$ such that
$$\lim_{k\To+\infty}x(t_k^n)=\ol x_n.$$

Further, $\lim_{n\To+\infty}\ol x_n=x^*,$ hence for every $\e>0$ there exists $N_\e\in\N$ such that for all $n\ge N_\e$ one has
$$\|\ol x_n-x^*\|<\frac\e2.$$
Let us fix $\e>0$ and let $n\ge 1$ fixed. Since $\lim_{k\To+\infty}x(t_k^n)=\ol x_n,$ there exists $k(n,\e)\in\N$ such that for all $k\ge k(n,\e)$ one has $$\|x(t_k^n)-\ol x_n\|<\frac\e2.$$
The  last two relations lead to
$$\|x(t_k^n)-x^*\|<\e,\,\forall n\ge N_\e,\,k\ge k(n,\e).$$

Since $t_k^n\To+\infty,\,k\To+\infty$ for all $n\ge 1$, we obtain that for every $n\ge 1$ there exists $N_n\in\N$ such that $t_k^n>n$ for all $k\ge N_n.$  For every $n\ge 1$ consider now $k_n>\max\{n, N_n,k(n,\e)\}$ and the sequence $s_n=t_{k_n}^n.$ Obviously $s_n\To+\infty,\,n\To+\infty$ and
$\|x(s_n)-x^*\|<\e,\,\forall n\ge N_\e.$ Hence
$$\lim_{n\To+\infty}x(s_n)=x^*\mbox{ and }s_n\To+\infty,\,n\To+\infty,$$
in other words $x^*\in \omega(x).$

(viii) We have
$$0\le\lim_{t\To+\infty}\dist((u(t),v(t)),\omega(u,v))\le\lim_{k\To+\infty}\dist((u(t_k),v(t_k)),\omega(u,v))=0.$$
\end{proof}

\begin{remark}\label{r} Combining (iii) and (iv) in Lemma \ref{reg}, one obtains that for every $\ol{x}\in\omega(x)$ and $t_k\To+\infty$ such that $x(t_k)\To\ol{x}$ as $k\To+\infty$ we have
$$\lim_{k\To+\infty}H(u(t_k),v(t_k))=H(\ol x,\ol x).$$
\end{remark}

The following result is the  main result of the paper.

\begin{theorem}\label{convergence}  Assume that $g$ is bounded from below and for $u_0,v_0\in\R^n$, let $x$  be the unique strong global solution of \eqref{dysy}. Consider the function
$$H:\R^n\times \R^n\To\R,\,H(x,y)=g(x)+\frac{1}{2}\|x-y\|^2.$$
Suppose that $x$ is bounded and $H$ is a KL function. Then the following statements are true
\begin{itemize}
\item[(a)] $\dot{x}(t)\in L^1([t_0,+\infty),\R^n);$
\item[(b)] $\frac{1}{t}\ddot{x}(t)\in L^1([t_0,+\infty),\R^n);$
\item[(c)] there exists $\ol{x}\in\crit(g)$ such that $\lim_{t\To+\infty}x(t)=\ol{x}.$
\end{itemize}
\end{theorem}
\begin{proof} Under the assumptions of Lemma \ref{reg}, consider $(\ol x,\ol x)\in\omega(u,v).$ Then

$$\lim_{t\To+\infty}H(u(t),v(t))=H(\ol x,\ol x).$$

{\bf Case I.} There exists $\ol t\ge t_1$ such that $H(u(\ol t),v(\ol t))=H(\ol x,\ol x).$ From Lemma \ref{reg} (ii) we have
$$\frac{d}{dt}H(u(t),v(t))\le A(t)\|\ddot{x}(t)\|^2+B(t)\|\dot{x}(t)\|^2\le 0,\,\mbox{ for all }t\ge t_1,$$
hence,
$$H(u(t),v(t))\le H(\ol x,\ol x),\,\mbox{ for all }t\ge \ol t.$$
On the other hand
$$H(u(t),v(t))\ge \lim_{t\To+\infty}H(u(t),v(t))=H(\ol x,\ol x),\,\mbox{ for all }t\ge t_1,$$
hence
$$H(u(t),v(t))= H(\ol x,\ol x),\,\mbox{ for all }t\ge \ol t.$$ But then, $H(u(t),v(t))$ is constant, hence $\frac{d}{dt}H(u(t),v(t))=0,\,\mbox{ for all }t\ge \ol t,$ which leads to
$$0\le A(t)\|\ddot{x}(t)\|^2+B(t)\|\dot{x}(t)\|^2\le 0,\,\mbox{ for all }t\ge \ol t.$$
Since $A(t)<0$ and $B(t)<0$ for every $t\ge t_1$, the latter inequality can hold only if
$$\dot{x}(t)=\ddot{x}(t)=0,\mbox{ for all }t\ge \ol t.$$
Consequently, $\dot{x},\ddot{x}\in L^1([t_0,+\infty),\R^n)$ and $x(t)=\ol x$ is constant on $[\ol t,+\infty).$

{\bf Case II.} For every $t\ge t_1$ one has that $H(u(t),v(t))> H(\ol x,\ol x).$
Let $\Omega=\omega(u,v).$ Then according to Lemma \ref{reg}, $\Omega$ is nonempty and compact and $H$ is constant on $\Omega.$
Since $H$ is KL, according to Lemma \ref{unif-KL-property}  there exist $\varepsilon,\eta >0$ and $\varphi\in \Theta_{\eta}$ such that for all $(\ol x,\ol x)\in\Omega$ and for all $(z,w)$ in the intersection
$$ \{(z,w)\in\R^n\times\R^n: \dist((z,w),\Omega)<\varepsilon\}\cap \{(z,w)\in\R^n\times\R^n: H(\ol x,\ol x)<H(z,w)<H(\ol x,\ol x)+\eta\}$$
one has
$$\varphi'(H(z,w)-H(\ol x,\ol x))\|\n H(z,w)\|\geq 1.$$

According to Lemma \ref{reg} (viii), $\lim_{t\To+\infty}\dist((u(t),v(t)),\Omega)=0$, hence there exists $t_2\ge 0$ such that
$$\dist((u(t),v(t)),\Omega)<\e,\,\forall t\ge t_2.$$

Since $$\lim_{t\To+\infty}\left(H(u(t),v(t))\right)= H(\ol{x},\ol{x})$$
and
$$H(u(t),v(t))> H(\ol{x},\ol{x}),$$
there exists $t_3>t_1$ such that
$$H(\ol x,\ol x)<H(u(t),v(t))< H(\ol{x},\ol{x})+\eta,\,\forall t\ge t_3.$$

Hence, for all $t\ge T=\max(t_2,t_3)$ we have
$$\varphi'(H(u(t),v(t))-H(\ol{x},\ol{x}))\cdot\|\n H(u(t),v(t))\|\ge 1.$$

 According to Lemma \ref{reg} (ii) and (v), we have $\frac{d}{dt}H(u(t),v(t))\le A(t)\|\ddot{x}(t)\|^2+B(t)\|\dot{x}(t)\|^2\le 0$
and $\|\nabla H(u(t),v(t))\|\le q(t)\|\ddot{x}(t)\|+p(t)\|\dot{x}(t)\|,$ hence,
$$\frac{d}{dt}\varphi(H(u(t),v(t))-H(\ol x,\ol x))=\varphi'(H(u(t),v(t))-H(\ol x,\ol x))\frac{d}{dt}H(u(t),v(t))\le\frac{A(t)\|\ddot{x}(t)\|^2+B(t)\|\dot{x}(t)\|^2}{q(t)\|\ddot{x}(t)\|+p(t)\|\dot{x}(t)\|}$$
for all $t\in [T,+\infty).$
By  integrating on the interval $[T,\ol t],\,\ol t>T$  we obtain
$$\varphi(H(u(\ol t),v(\ol t))-H(\ol x,\ol x))- \int_T^{\ol t} \frac{A(s)\|\ddot{x}(s)\|^2+B(s)\|\dot{x}(s)\|^2}{q(s)\|\ddot{x}(s)\|+p(s)\|\dot{x}(s)\|}ds\le\varphi(H(u(T),v(T))-H(\ol x,\ol x)).$$

Since $\varphi$ is bounded from below, $A(s)<0,B(s)<0,p(s)>0,q(s)>0$ for all $s\ge T$ and $\ol t$ was arbitrary chosen, we obtain that
$$0\le\int_T^{+\infty} \frac{-A(s)\|\ddot{x}(s)\|^2-B(s)\|\dot{x}(s)\|^2}{q(s)\|\ddot{x}(s)\|+p(s)\|\dot{x}(s)\|}ds\le\varphi(H(u(T),v(T))-H(\ol x,\ol x))$$
which leads to
$$ \frac{\frac{1}{t^3}\|\ddot{x}(t)\|^2}{q(t)\|\ddot{x}(t)\|+p(t)\|\dot{x}(t)\|},\frac{\frac{1}{t}\|\dot{x}(t)\|^2}{q(t)\|\ddot{x}(t)\|+p(t)\|\dot{x}(t)\|}\in
 L^1([t_0,+\infty),\R^n).$$

By using the arithmetical-geometrical mean inequality we have
 $$\frac12\left(\frac{\frac{1}{t^3}\|\ddot{x}(t)\|^2}{q(t)\|\ddot{x}(t)\|+p(t)\|\dot{x}(t)\|}+\frac{\frac{1}{t}\|\dot{x}(t)\|^2}{q(t)\|\ddot{x}(t)\|+p(t)\|\dot{x}
 (t)\|}\right)\ge
 \frac{\frac{1}{t^2}\|\ddot{x}(t)\|\|\dot{x}(t)\|}{q(t)\|\ddot{x}(t)\|+p(t)\|\dot{x}(t)\|},$$
 hence,
 $$ \frac{\frac{1}{t^2}\|\ddot{x}(t)\|\|\dot{x}(t)\|}{q(t)\|\ddot{x}(t)\|+p(t)\|\dot{x}(t)\|}\in L^1([t_0,+\infty),\R^n).$$

 Since $\mathcal{O}\left(\frac{q(t)}{t}\right)=\mathcal{O}\left(\frac{1}{t^3}\right),\,\mathcal{O}(p(t))=\mathcal{O}\left(\frac{1}{t}\right),$ and
 $\mathcal{O}\left(q(t)+\frac{p(t)}{t}\right)=\mathcal{O}\left(\frac{1}{t^2}\right),$ one has
 $$\frac{\frac{q(t)}{t}\|\ddot{x}(t)\|^2}{q(t)\|\ddot{x}(t)\|+p(t)\|\dot{x}(t)\|}+\frac{p(t)\|\dot{x}(t)\|^2}{q(t)\|\ddot{x}(t)\|+p(t)\|\dot{x}(t)\|}+
\frac{\left(q(t)+\frac{p(t)}{t}\right)\|\ddot{x}(t)\|\|\dot{x}(t)\|}{q(t)\|\ddot{x}(t)\|+p(t)\|\dot{x}(t)\|}=$$
$$\|\dot{x}(t)\|+\frac1t\|\ddot{x}(t)\|,$$
which  shows that $$\dot{x}(t),\frac{1}{t}\ddot{x}(t)\in L^1([t_0,+\infty),\R^n).$$

Now, since  $\dot{x}(t)\in L^1([t_0,+\infty),\R^n)$, it follows that there exists and is finite the limit $\lim_{t\To+\infty}x(t).$ But, obviously, then one must have
$$\lim_{t\To+\infty}x(t)=\ol x.$$
\end{proof}

\begin{corollary}\label{sa}Assume that $g$ is semi-algebraic and bounded from below and for $u_0,v_0\in\R^n$, let $x$  be the unique strong global solution of \eqref{dysy}. If $x$ is bounded then all the conclusions of Theorem \ref{convergence} hold.
\end{corollary}
\begin{proof}
  We need only to prove that $H$ is a KL function. Since the class of semi-algebraic functions is closed under addition (see for example \cite{b-sab-teb}) and
$(x,y) \mapsto \frac12\|x-y\|^2$ is semi-algebraic, we obtain that the  function
$$H:\R^m\times\R^m\To\R\cup\{+\infty\},\,H(x,y)=g(x)+\frac12\|y-x\|^2$$ is semi-algebraic. Consequently, $H$ is a KL function.
\end{proof}

\begin{remark}\label{cond-x-bound}
Note that in the hypotheses of Theorem \ref{convergence} one need to assume that the  trajectory $x$ generated by \eqref{dysy} is bounded and $g$ is  bounded from below.

In order to validate these assumptions it is enough to assume that  $g$ is coercive,
that is $$\lim_{\|u\|\rightarrow+\infty}(g)(u)=+\infty.$$

Indeed, notice that $g$ is bounded from below, being a continuous and coercive function (see for example \cite{rock-wets}). From \eqref{e12} it follows that $\frac{\b}{T}\dot{x}(T)+x(T)$ is contained for every $T\geq t_1$ in a lower level set
of $g$, which is a bounded set due to the coercivity assumption. But, according to Theorem \ref{abstr}, $\frac{\b}{t}\dot{x}(t)\To 0,\,t\To+\infty$, hence one can easily deduce that $x$  is bounded.
\end{remark}

\section{Convergence rates}\label{sec5}

In the context of optimization problems involving KL functions, it is known (see \cite{lojasiewicz1963, b-d-l2006, attouch-bolte2009})
that convergence rates of the trajectory can be formulated in terms of the  \L{}ojasiewicz exponent.

In the following theorem we obtain convergence rates for the trajectory generated by \eqref{dysy}.

\begin{theorem}\label{th-conv-rate}  Assume that $g$ is bounded from below and for $u_0,v_0\in\R^n$, let $x$  be the unique strong global solution of \eqref{dysy}. Consider the function
$$H:\R^n\times \R^n\To\R,\,H(x,y)=g(x)+\frac{1}{2}\|x-y\|^2.$$
Suppose that $x$ is bounded and let $\ol x\in\crit(g)$ be such that $\lim_{t\To+\infty}x(t)=\ol x$ and $H$ fulfills the {\L}ojasiewicz property at $(\ol x,\ol x)\in\crit H$ with {\L}ojasiewicz exponent  $\t$.

Then, there exist $a_1,a_2,a_3,a_4>0$ and $t'>0$ such that for every $t \in [t', +\infty)$ the following statements are true
\begin{itemize}
\item[(a)] if $\t\in(0,\frac12),$ then $x$ converges in finite time;
\item[(b)] if $\t=\frac12,$ then $\|x(t)-\ol x\|\le a_1 e^{-a_2t^2}$;
\item[(c)] if $\t\in(\frac12,1),$ then $\|x(t)-\ol x\|\le (a_3t^2+a_4)^{-\frac{1-\t}{2\t-1}}.$
\end{itemize}
\end{theorem}

\begin{proof}
We define for every $t\in [t_1, +\infty)$, ($t_1$ is defined in the proof of Theorem \ref{abstr}),
$$\s(t):=\int_t^{+\infty}\left(\|\dot{x}(s)\|+\frac1s\|\ddot{x}(s)\|\right)ds.$$
Let $t\in [t_1, +\infty)$ be fixed. For $T\ge t$ we have
$$\|x(t)-\ol x\|=\left \|x(T)-\ol x-\int_t^{T}\dot{x}(s)ds \right\|\le\|x(T)-\ol x\|+\int_t^{T}\|\dot{x}(s)\|ds.$$
By taking the limit as  $T\To+\infty$ we obtain
\begin{equation}\label{e17}
\|x(t)-\ol x\|\le\int_t^{+\infty}\|\dot{x}(s)\|ds\le\s(t).
\end{equation}

We  show next, that there exists $m<0$ such that
\begin{equation}\label{e18}
 \frac{A(t)\|\ddot{x}(t)\|^2+B(t)\|\dot{x}(t)\|^2}{q(t)\|\ddot{x}(t)\|+p(t)\|\dot{x}(t)\|}\le m\left(\|\dot{x}(t)\|+\frac1t\|\ddot{x}(t)\|\right).
 \end{equation}

Indeed,
$$(q(t)\|\ddot{x}(t)\|+p(t)\|\dot{x}(t)\|)\left(\|\dot{x}(t)\|+\frac1t\|\ddot{x}(t)\|\right)=\frac{q(t)}{t}\|\ddot{x}(t)\|^2+\left(q(t)+\frac{p(t)}{t}\right)\|\dot{x}(t)\|\|\ddot{x}(t)\|+
p(t)\|\dot{x}(t)\|^2=$$
$$\frac{q(t)}{t}\|\ddot{x}(t)\|^2+\left(tq(t)+p(t)\right)\|\dot{x}(t)\|\left\|\frac1t\ddot{x}(t)\right\|+
p(t)\|\dot{x}(t)\|^2\le$$
$$\left(\frac{q(t)}{t}+\frac{tq(t)+p(t)}{2t^2}\right)\|\ddot{x}(t)\|^2+\left(p(t)+\frac{tq(t)+p(t)}{2}\right)\|\dot{x}(t)\|^2=$$
$$\left(\frac32\frac{q(t)}{t}+\frac{p(t)}{2t^2}\right)\|\ddot{x}(t)\|^2+\left(\frac32p(t)+\frac{tq(t)}{2}\right)\|\dot{x}(t)\|^2\le$$
$$\frac{A(t)}{m}\|\ddot{x}(t)\|^2+\frac{B(t)}{m}\|\dot{x}(t)\|^2,$$
where
$$m=\max\left(\max_{t\ge t_1}\frac{A(t)}{\frac32\frac{q(t)}{t}+\frac{p(t)}{2t^2}},\max_{t\ge t_1}\frac{B(t)}{\frac32p(t)+\frac{tq(t)}{2}}\right).$$
It is an easy verification that $m<0.$

As we have seen in the proof of Theorem \ref{convergence} if there exists $\ol t\ge t_1$ such that $H(u(\ol t),v(\ol t))=H(\ol x,\ol x)$ then $x$ is constant on $[\ol t,+\infty),$ therefore the conclusion follows.

On the other hand, if for every $t\ge t_1$ one has that $H(u(t),v(t))> H(\ol x,\ol x)$ then according to the proof of Theorem \ref{abstr} (i) and the fact that $\lim_{t\To+\infty}x(t)=\ol x$, there exists $\e>0$ and $t'\ge t_1$ such that
$$\|(u(t),v(t))-(\ol x,\ol x)\|<\e.$$
Further, according to the proof of Theorem \ref{convergence} one has
$$K\frac{d}{dt}(H(u(t),v(t))-H(\ol x,\ol x))^{1-\t}\le \frac{A(t)\|\ddot{x}(t)\|^2+B(t)\|\dot{x}(t)\|^2}{q(t)\|\ddot{x}(t)\|+p(t)\|\dot{x}(t)\|},\mbox{ for all }t\ge t'.$$
Now by using \eqref{e18} we obtain that
$$M\left(\|\dot{x}(t)\|+\frac1t\|\ddot{x}(t)\|\right)+\frac{d}{dt}(H(u(t),v(t))-H(\ol x,\ol x))^{1-\t}\le 0,\mbox{ for all }t\ge t',$$
where $M=-\frac{m}{K}>0.$

We integrate the last relation on the interval $[t,T],\,T>t\ge t'$ and we obtain
$$M\int_t^T\|\dot{x}(s)\|+\frac1s\|\ddot{x}(s)\|ds+(H(u(T),v(T))-H(\ol x,\ol x))^{1-\t}\le (H(u(t),v(t))-H(\ol x,\ol x))^{1-\t}.$$
Consequently, for $T\To+\infty$ we have
$$M\s(t)\le (H(u(t),v(t))-H(\ol x,\ol x))^{1-\t},\mbox{ for all }t\ge t'.$$
But, $$(H(u(t),v(t))-H(\ol x,\ol x))^{\t}\le K\|\n H(u(t),v(t)\|$$
and, according to Lemma \ref{reg} (v),
$$\|\n H(u(t),v(t)\|\le q(t)\|\ddot{x}(t)\|+p(t)\|\dot{x}(t)\|,\mbox{ for all }t\ge t',$$
hence
$$M\s(t)\le K^{\frac{1-\t}{\t}}(q(t)\|\ddot{x}(t)\|+p(t)\|\dot{x}(t))^{\frac{1-\t}{\t}},\mbox{ for all }t\ge t'.$$
Easily can be checked that there exists $a>0$ such that
$$q(t)\|\ddot{x}(t)\|+p(t)\|\dot{x}(t)\le\frac{a}{t}\left(\dot{x}(t)\|+\frac1t\|\ddot{x}(t)\|\right),\mbox{ for all }t\ge t',$$
hence
$$M\s(t)\le \left(\frac{aK}{t}\right)^{\frac{1-\t}{\t}}\left(\dot{x}(t)\|+\frac1t\|\ddot{x}(t)\|\right)^{\frac{1-\t}{\t}},\mbox{ for all }t\ge t'.$$
But $\dot{x}(t)\|+\frac1t\|\ddot{x}(t)\|=-\dot{\s}(t)$, consequently
\begin{equation}\label{e19}
-\a t\s^{\frac{\t}{1-\t}}(t)\ge\dot{\s}(t),\mbox{ for all }t\ge t',
\end{equation}
where $\a=\frac{M^\frac{\t}{1-\t}}{aK}>0.$

If $\t=\frac12$ then \eqref{e19} becomes $\a t\s(t)+\dot{\s}(t)\le0,\mbox{ for all }t\ge t'.$ By  multiplying with $e^{\frac{\a t^2}{2}}$ then integrating on $[t',t]$ we get that there exists $a_1,a_2>0$ such that
$$\s(t)\le a_1 e^{-a_2 t^2},\mbox{ for all }t\ge t'.$$
Now using \eqref{e17} we obtain
$$\|x(t)-\ol x\|\le a_1 e^{-a_2 t^2},\mbox{ for all }t\ge t',$$
which  proves (b).

Assume now that $0<\t<\frac12.$  We prove (a) by contradiction. So, assume that for all $t\ge t'$ one has $\s(t)\neq 0.$
By using \eqref{e19} again we obtain
$$\frac{d}{dt}\s^{\frac{1-2\t}{1-\t}}(t)=\frac{1-2\t}{1-\t}\s^{\frac{-\t}{1-\t}}(t)\dot{\s}(t)\le-\a\frac{1-2\t}{1-\t}t,\mbox{ for all }t\ge t'.$$
By integrating on $[t',t]$ we obtain
$$\s^{\frac{1-2\t}{1-\t}}(t)\le-\ol\a t^2+\ol\b,\mbox{ for all }t\ge t',$$
where $\ol\a>0.$  Hence, $\s(t)\To -\infty$ as $t\To+\infty,$ which is in contradiction with $\s(t)\ge 0$ for all $t\ge t'.$
It is thus deduced by contradiction that there exists $T \ge t'$ such that $\s(T) = 0.$ Since, according to \eqref{e19}, $\s$ is
nonnegative and nonincreasing, one has $\s(t) = 0$ for all $t \ge T,$ therefore  \eqref{e17}  implies that $x$ is constant on $[T,+\infty)$ and (a) follows.

Assume now that $\frac12<\t<1.$ By using \eqref{e19}  we obtain
$$\frac{d}{dt}\s^{\frac{1-2\t}{1-\t}}(t)=\frac{1-2\t}{1-\t}\s^{\frac{-\t}{1-\t}}(t)\dot{\s}(t)\ge\a\frac{2\t-1}{1-\t}t,\mbox{ for all }t\ge t'.$$
By integrating on $[t',t]$ we obtain
$$\s^{\frac{1-2\t}{1-\t}}(t)\ge a_3 t^2+a_4,\mbox{ for all }t\ge t',$$
where $a_3,a_4>0,$ or equivalently,
$$\s(t)\le (a_3 t^2+a_4)^{-\frac{1-\t}{2\t-1}},\mbox{ for all }t\ge t'.$$
Now using \eqref{e17} we obtain
$$\|x(t)-\ol x\|\le (a_3 t^2+a_4)^{-\frac{1-\t}{2\t-1}},\mbox{ for all }t\ge t'$$
and this proves (c).
\end{proof}

According to Theorem 3.6 \cite{LiP}, if $g$ has the KL property with KL exponent $\t\in\left[\frac12,1\right)$ at $\ol x\in\R^m,$ then the function $H:\R^m\times\R^m\To\R\cup\{+\infty\},\,H(x,y)=g(x)+\frac12\|y-x\|^2$ has the KL property at $(\ol x,\ol x)\in\R^m\times\R^m$ with the same KL exponent $\t.$ This result allows us to reformulate Theorem \ref{th-conv-rate}.

\begin{corollary}\label{cuse}   Assume that $g$ is bounded from below and for $u_0,v_0\in\R^n$, let $x$  be the unique strong global solution of \eqref{dysy}.
Suppose that $x$ is bounded and let $\ol x\in\crit(g)$ be such that $\lim_{t\To+\infty}x(t)=\ol x$ and $g$ fulfills the {\L}ojasiewicz property at $\ol x\in\crit g$ with {\L}ojasiewicz exponent   $\t\in\left[\frac12,1\right)$. If $\t=\frac12$ then the convergence rates stated at {\rm Theorem \ref{th-conv-rate}(b)} hold,  if  $\t\in\left(\frac12,1\right)$  then the convergence rates stated at  {\rm Theorem \ref{th-conv-rate}(c)} hold.
\end{corollary}
\begin{proof}
 Indeed, according to Theorem 3.6 \cite{LiP}
$H$ has the {\L}ojasiewicz property at $(\ol x,\ol x)\in\crit H$ with the {\L}ojasiewicz exponent   $\t\in\left[\frac12,1\right)$. Hence, Theorem \ref{th-conv-rate} can be applied.
\end{proof}

In case we assume that the function $g$ is strongly convex, then Theorem \ref{th-conv-rate} assures superlinear convergence rates for the trajectories generated by \eqref{dysy}. The following result holds.

\begin{theorem}\label{ratestrconv} Assume that the objective function $g$ is strongly convex and let $x^*$ be the unique minimizer of $g.$ For $u_0,v_0\in\R^n$, let $x$  be the unique strong global solution of \eqref{dysy}.

Then, there exist $a_1,a_2>0$ and $t'>0$ such that for every $t \in [t', +\infty)$ one has
 $$\|x(t)-x^*\|\le a_1 e^{-a_2t^2}.$$
\end{theorem}
\begin{proof}
We emphasize that the strongly convex function $g$  is coercive, see \cite{BauComb}. According to \cite{rock-wets} the function $g$ is bounded from bellow. Hence, $\omega(x)=\{x^*\}$ and $x(t)\To x^*,\,t\To+\infty.$ According to \cite{attouch-bolte2009}, $g$ satisfies the Kurdyka-{\L}ojasiewicz property at $x^*$ with the {\L}ojasiewicz exponent $\t=\frac12.$ The conclusion now follows from Corollary \ref{cuse}.
\end{proof}

\vskip 6mm
\noindent{\bf Acknowledgements}
The author is thankful to an anonymous referee for his/her valuable remarks and suggestions which improved the quality of the paper.
\noindent
This work was supported by a grant of the Ministry of Research, Innovation and Digitization, CNCS -
UEFISCDI, project number PN-III-P1-1.1-TE-2021-0138, within PNCDI III.

\end{document}